\documentclass[a4paper,11pt,reqno]{amsart}

\usepackage{latexsym,dsfont,amssymb,amsmath,amsthm,color}

\chardef\bslash=`\\ % p. 424, TeXbook

\numberwithin{equation}{section}

\newtheorem{theorem}{Theorem}[section]
\newtheorem{corollary}[theorem]{Corollary}
\newtheorem{lemma}[theorem]{Lemma}
\newtheorem{proposition}[theorem]{Proposition}

\theoremstyle{remark}
\newtheorem{remark}[theorem]{Remark}

\theoremstyle{definition}

\newcommand\bp{\begin{proof}}
\newcommand\ep{\end{proof}}

\newcommand{\Zhat}{\widehat\Z}

\newcommand{\nx}{\mathbb N^{\times}}

\newcommand{\nxnx}{{\mathbb N \rtimes \mathbb N^\times}}
\newcommand{\qxqx}{{\mathbb Q \rtimes \mathbb Q^*_+}}
\newcommand{\N}{\mathbb N}

\newcommand{\Z}{\mathbb Z}
\newcommand{\Q}{\mathbb Q}

\newcommand{\C}{\mathbb C}
\newcommand{\R}{\mathbb R}
\newcommand{\T}{\mathbb T}
\newcommand{\TT}{\mathcal T}

\newcommand{\primes}{\mathcal P}

\newcommand{\SSS}{\mathcal S}

\newcommand\enu[1]{\smallskip\newline\makebox[5mm][l]{\rm(#1)}}

\newcommand\chf{{\mathds 1}}
\newcommand\af{\mathbb{A}_f}
\newcommand\gl{{\operatorname{GL}_2}}

\begin{document}

\title[Type III$_1$ equilibrium states]
{Type III$_1$ equilibrium states of the Toeplitz algebra of the affine semigroup over the natural numbers}
%{Classification of high-temperature KMS states of the Toeplitz algebra of $\nxnx$}

\author[Laca]{Marcelo Laca}
\author[Neshveyev]{Sergey Neshveyev}

\address{Marcelo Laca, Department of Mathematics and Statistics\\
University of Victoria\\
PO Box 3060\\
Victoria, BC V8W 3R4\\
Canada}
\email{laca@math.uvic.ca}

\address{Sergey Neshveyev,  Department of Mathematics\\
University of Oslo\\
PO Box 1053 Blindern\\
N-0316 Oslo\\
Norway}
\email{sergeyn@math.uio.no}

\begin{abstract}
We complete the analysis of KMS-states of
the Toeplitz algebra $ \TT(\nxnx) $ of the affine semigroup  over the natural numbers,
recently studied by Raeburn and the first author, by showing that for every inverse temperature
$\beta$ in the critical interval $1\leq \beta \leq 2$, the unique KMS$_\beta$-state is of type~III$_1$.
We prove this by reducing the type classification from $\TT(\nxnx)$
to that of the symmetric part of the Bost-Connes system, with a shift in inverse temperature.
To carry out this reduction we first obtain a parametrization of the Nica spectrum of $\nxnx$
in terms of an adelic space. Combining a characterization of traces on crossed products due to the second author with an analysis of the action of~$\nxnx$ on the Nica spectrum, we can also recover all the KMS-states of $\TT(\nxnx)$ originally computed by Raeburn and the first author. Our computation sheds light on why there is a free transitive circle action on the extremal KMS$_\beta$-states for $\beta >2$ that does not ostensibly come from an action of $\T$ on the C$^*$-algebra.
\end{abstract}

\date{October 4, 2010}

\thanks{This research has been supported by the Natural Sciences and Engineering Research Council of Canada and by the Research Council of Norway.}

\maketitle

\section*{Introduction}

The Toeplitz algebra $\TT(\nxnx)$ of the semigroup $\nxnx$ was recently studied in
\cite{LR10} with the double motivation of realizing Cuntz's algebra $\mathcal Q_\N$~\cite{Cu} as a boundary quotient and of exploring the structure of KMS-states with respect to a natural dynamics extending that on~$\mathcal Q_\N$. Cuntz showed that $\mathcal Q_\N$ has a unique KMS-state, at inverse temperature $\beta=1$ and, in view of \cite{LN}, it was therefore natural to expect that $\TT(\nxnx)$ would have a KMS$_\beta$-state for every $\beta\ge1$ and no KMS$_\beta$-states for $\beta<1$. This turned out to be true; in fact, as was shown in \cite{LR10}, the system has a phase transition at $\beta=2$: for every $\beta\in[1,2]$ there exists a unique KMS$_\beta$-state, while for $\beta>2$ the simplex of KMS$_\beta$-states is isomorphic to the simplex of probability measures on $\T$. It was also shown in~\cite{LR10} that every extremal KMS$_\beta$-state for $\beta>2$ is of type~I. Our primary goal in this paper is to complete the analysis of KMS-states by showing that the KMS$_\beta$-states for $\beta\in[1,2]$ are of type~III$_1$. Along the way we answer several other natural questions about the C*-algebra $\TT(\nxnx)$ and dynamical systems associated with it, and we obtain a characterization of KMS-states of crossed products in terms of traces that might be of further use elsewhere.

\smallskip

In Section~\ref{s1} we begin by analyzing the Nica spectrum $\Omega$ of $\nxnx$. This is a compact space with an action of $\nxnx$ by injective maps such that $\TT(\nxnx)=C(\Omega)\rtimes(\nxnx)$. A description of this space in terms of supernatural and adic numbers was given in~\cite{LR10}. Starting from that description we show that $\Omega$ can be identified with the disjoint union of $\N\times(\Zhat/\Zhat^*)$ and a quotient of $\Zhat\times(\Zhat/\Zhat^*)$, where $\Zhat=\prod_p\Z_p$ is the compact ring of finite integral adeles, and $\hat \Z^*$ is the group of invertible elements, the integral ideles. The topology on $\Omega$ is obtained by considering $\N\sqcup\Zhat$ as a compactification of the discrete set $\N$ with boundary $\Zhat$. The action of $\nxnx$ on $\Omega$ in this picture is defined using the obvious actions of $\nxnx$ on $\N\times(\Zhat/\Zhat^*)$ and $\Zhat\times(\Zhat/\Zhat^*)$, given by the formula $(m,k)(r,a)=(m+kr,ka)$. This allows us to answer easily various questions about this action arising in the study of $\TT(\nxnx)$, such as which orbits are dense or which points have nontrivial stabilizers. Furthermore, this also allows us to construct an explicit dilation of the action of~$\nxnx$ on $\Omega$ to an action of $\qxqx$ on a locally compact space $\tilde\Omega$, thus realizing $\TT(\nxnx)$ as a full corner in $C_0(\tilde\Omega)\rtimes(\qxqx)$.

\smallskip

In Section~\ref{secphasetrans} we discuss the classification of KMS-states of $\TT(\nxnx)$ obtained in~\cite{LR10} in terms of measures on our parameter space. We first consider states which factor through the conditional expectation onto $C(\Omega)$ and thus are determined by probability measures on $\Omega$. We show that such a measure defines a KMS$_\beta$-state if and only if it is the push-forward of the measure $\nu_1\times\nu$ on $\Zhat\times(\Zhat/\Zhat^*)$, where $\nu_1$ is the Haar measure on~$\Zhat$ and $\nu$ is a measure defining a KMS$_{\beta-1}$-state on the symmetric part $\TT(\N^\times)=C(\Zhat/\Zhat^*)\rtimes\N^\times$ of the Bost-Connes system from~\cite{bos-con}. In particular, for every $\beta\ge1$ there exists a unique such measure $\mu_\beta$ on $\Omega$, and there are no such measures for $\beta<1$. This can of course be deduced from~\cite{LR10}, but we prove it from scratch, demonstrating how the symmetric part of the Bost-Connes system can be used to analyze $\TT(\nxnx)$. It turns out that for $\beta\le2$ every KMS$_\beta$-state on $\TT(\nxnx)$ factors through the conditional expectation onto~$C(\Omega)$, so in this case the classification of KMS-states can be deduced from known facts about the Bost-Connes system.

We then turn to $\beta>2$. In this case it is already known from~\cite{LR10} that there are KMS$_\beta$-states which do not factor through the conditional expectation onto~$C(\Omega)$; specifically, the extreme points of the simplex of KMS$_\beta$-states are indexed by $\T$. The reason for this, from our perspective, is that in this case the measure $\mu_\beta$ is concentrated on a countable set of points with nontrivial stabilizers in~$\qxqx$. By~\cite{nes3} it follows that the state $\mu_{\beta*}$ on~$C(\Omega)$ defined by~$\mu_\beta$ can be extended to a state on $\TT(\nxnx)$ with centralizer containing~$C(\Omega)$ by choosing arbitrary states on the group C$^*$-algebras of the stabilizers. It turns out that all the
relevant stabilizers are isomorphic to~$\Z$ and the corresponding points lie all on the same $(\qxqx)$-orbit. Thus,  to extend $\mu_{\beta*}$ we need a countable collection of probability measures on~$\T$, but since $\qxqx$ acts transitively on our set of points,  to get a KMS-state we can choose an arbitrary probability measure on $\T$ for just one point and then the KMS-condition gives measures for all other points. This gives an explanation, from the point of view of  dynamical systems,  of the existence of a free transitive circle action on the extremal KMS$_\beta$-states observed in~\cite{LR10}. Note that the formal argument in Section~\ref{secphasetrans} is slightly different from the one described here, but the idea is the same.

\smallskip

In Section~\ref{stype} we prove our main result: if $\varphi_\beta$ is the unique KMS$_\beta$-state on $\TT(\nxnx)$ for $\beta\in[1,2]$, then $\pi_{\varphi_\beta}(\TT(\nxnx))''$ is the hyperfinite factor of type III$_1$. For $\beta=1$ this is straightforward, while for $\beta\in(1,2]$ we show that the flow of weights for $\pi_{\varphi_\beta}(\TT(\nxnx))''$ can be identified with that for the von Neumann algebra generated by the symmetric part of the Bost-Connes system in the GNS-representation corresponding to the unique KMS-state at inverse temperature $\beta-1$.

\smallskip

Finally, in Section~\ref{scrossed} we state and prove an auxiliary result, needed in Section~\ref{secphasetrans}, characterizing KMS-states on crossed products $A\rtimes\SSS$ by abelian semigroups in terms of traces on $A$. As should be obvious from the
general context of this section, the results presented here are independent from the rest of the paper.

\medskip

\noindent{\bf Acknowledgement.} The authors are grateful to Iain Raeburn for several conversations at the initial stage of the project.

\bigskip

\section{Crossed product decompositions} \label{s1}

Let $\N\rtimes\N^\times$ be the semidirect product of the additive semigroup $\N = \{0, 1, 2, \cdots\}$ by the multiplicative semigroup $\nx := \N \setminus \{0\}$.
The Toeplitz algebra of $\N\rtimes\N^\times$ is the C$^*$-subalgebra $\TT(\N\rtimes\N^\times)$ of $B(\ell^2(\N\rtimes\N^\times))$ generated by the operators $T_x$ defined by $T_x\delta_y=\delta_{xy}$, for $x,y\in\N\rtimes\N^\times$. Let $s=T_{(1,1)}$ be the isometry corresponding to the additive generator of $\N$ and, for every prime number $p\in\primes$, let $v_p=T_{(0,p)}$. Then $\TT(\N\rtimes\N^\times)$ is generated by the isometries $s$ and $v_p$, $p\in\primes$.

\smallskip

It turns out that the semigroup $\N\rtimes\N^\times$ determines a quasi-lattice order on the group  $\Q\rtimes\Q^*_+$ \cite[Proposition 2.2]{LR10} and thus the Toeplitz algebra is canonically isomorphic to a semigroup crossed product. We  restrict ourselves to a brief account here and we refer to  \cite{LR10} for the details.
The quasi-lattice property implies that the collection of
characteristic functions of the sets $x(\N\rtimes\N^\times)$ for $x\in\N\rtimes\N^\times$ is closed under multiplication and thus their closed linear span is a C*-subalgebra of
$\ell^\infty(\N\rtimes \N^\times)$. According to a general result of Nica~\cite{Ni}, the spectrum $\Omega$ of  this subalgebra is homeomorphic to the space of nonempty hereditary directed subsets of the semigroup. There is an action of~$\N\rtimes\N^\times$ on projections in $\ell^\infty(\N\rtimes\N^\times)$: the image of the characteristic function of a  set $A$ under~$x$ is the characteristic function of the set $xA$.
This induces an action of $\nxnx$ by endomorphisms of~$C(\Omega)$ and also a corresponding
action of $\nxnx$ by injective maps on $\Omega$, so that $x\in\nxnx$ maps $f\in C(\Omega)$ into $f(x^{-1}\cdot)$; here the convention is that $f(x^{-1}\omega)=f(\omega')$ if $\omega=x\omega'$, and $f(x^{-1}\omega)=0$ if $\omega\notin x\Omega$.
Since $\qxqx$ is amenable,  the results of \cite{Ni} and  \cite{LR96}
imply that  $\TT(\N\rtimes\N^\times)$ is canonically isomorphic to the
semigroup crossed product
$$
C(\Omega)\rtimes(\N\rtimes\N^\times),
$$
which, by definition, is the universal unital C$^*$-algebra generated by the image of a unital $*$-homomorphism $\iota\colon C(\Omega)\to C(\Omega)\rtimes(\N\rtimes\N^\times)$ and isometries $v_x$, for $x\in\N\rtimes\N^\times$, such that
$$
v_xv_y=v_{xy}\ \ \hbox{and}\ \ v_x\iota(f)v_x^*=\iota(f(x^{-1}\cdot)).
$$

\smallskip

In \cite[Corollary 5.6]{LR10} a  parametrization of the
space $\Omega$ of hereditary directed subsets of $\nxnx$ is given
in terms of supernatural numbers $N\in\mathcal N$ and $N$-adic numbers $r\in \Z/N: =\varprojlim \big((\Z/n\Z):n\in\nx,\;n|N\big)$.
We will realize the parameter space as a quotient of an adelic space that
carries a natural topology; this will make the  topology of $\Omega$ and the action of
$\nxnx$ more transparent and easy to work with.

Let $\Zhat=\prod_{p\in\primes}\Z_p$ be the compact ring  of finite integral adeles.
We begin by defining a topology on the disjoint union $\bar\N:=\N\sqcup\Zhat$
that turns it into a compactification of $\N$.
Denote by~$\gamma$ the canonical diagonal embedding $\N\hookrightarrow\Zhat$. The base of the topology on $\bar\N$ consists of singletons in~$\N$ and sets of the form $\{n\geq n_0\mid \gamma(n)\in U\}\sqcup U$, where $n_0\in\N$ and $U$ is an open subset of $\Zhat$ in its usual topology. In particular, a sequence converges to a point in $\N$
if and only if it is eventually constant and a sequence in $\Zhat$ converges to a point in $\Zhat$
if and only if it does in the usual topology, but
a sequence $\{m_n\}_n\subset\N\subset\bar\N$ converges to an element $r\in\Zhat\subset\bar\N$ if and only if $m_n\to+\infty$ in the usual sense and $\gamma(m_n)\to r$ in $\Zhat$.

The finite integral ideles $\Zhat^* := \prod_p \Z_p^*$ act by multiplication on $\Zhat$,
and every orbit  has a unique representative $a = (a_p)_{p\in \primes}\in\Zhat$ such that for every $p\in\primes$ we have $a_p=p^n$ for some $n\in\N\cup\{\infty\}$, with the convention that $p^\infty=0$. Therefore the orbit space
$\Zhat/\Zhat^*$ can be thought of as the set~$\mathcal N$ of supernatural numbers. Since  $\Zhat^*$ is a compact group, $\Zhat/\Zhat^*$  is a compact Hausdorff space. This corresponds to viewing the supernatural numbers as the product over all primes $p\in \primes$
of the one-point compactifications $p^\N \sqcup \{0\}$.

For every $a\in\Zhat$ denote by $\bar\N_a$ the quotient of $\bar\N$
obtained by identifying points in $\Zhat$ that are equal modulo $a\Zhat$, in other words,
$\bar\N_a$ is the
disjoint union $\N\sqcup(\Zhat/a\Zhat)$; in particular  $\bar\N$ is $\bar\N_0$. Note that $\Zhat/a\Zhat$ and hence $\bar\N_a$ depend only on the image of $a$ in $\Zhat/\Zhat^*$, and
that for each $a\in\Zhat/\Zhat^*$, the quotient $\Zhat/ a\Zhat$ is naturally isomorphic
%\footnote{We might want to sketch a proof of this claim.}
to the $a$-adic numbers $\Z/a=\varprojlim \big((\Z/n\Z):n\in\nx,\;n|a\big)$ from  \cite[Section 1.2]{LR10}.

To obtain our parametrization of the spectrum, we consider the map $p\colon  \bar\N \times (\Zhat/\Zhat^*) \to \Omega $  defined as follows.
Suppose $(r,a) \in \bar\N \times \Zhat/\Zhat^*$.
If $r\in\N\subset\bar\N$, we say that $(r,a)$ is of type A, and we let
\begin{equation} \label{A-type}
p(r,a) = A(r,a) :=\{(m,k)\in\nxnx\mid a\in k\Zhat/\Zhat^*\ \hbox{and}\ r-m\in k\N\}.
\end{equation}
If $r\in\Zhat\subset\bar \N$, we say that $(r,a)$ is of type~B; in this case we
denote by $r_a$ the class of $r$ in $\Zhat /a\Zhat$ and we let
\begin{equation}\label{B-type}
p(r,a) = B(r_a,a) :=\{(m,k)\in\nxnx\mid a\in k\Zhat/\Zhat^*\ \hbox{and}\ r-m\in k\Zhat\}.
\end{equation}
Note that if $a\in k\Zhat/\Zhat^*$ then $a\Zhat\subset k\Zhat$, so the set $p(r,a)$ indeed depends only on the class of $r$ in~$\Zhat/a\Zhat$.

These are the sets of type A and B appearing in~\cite[Proposition 5.1]{LR10}. As we have already observed, the map $p$ is not injective because in the second case, when $(r,a)$ is of type B,  the image $p(r,a)$ depends on $r$ only through the class $r_a= r+a\Zhat$. Identify points of type B in $\bar\N \times \Zhat/\Zhat^*$
according to the equivalence relation $(r,a) \sim (r',a')$ iff $a = a'$ and $r - r' \in a\Zhat$,
and consider the quotient space
$$ (\bar\N \times \Zhat/\Zhat^*)/_\sim = \{(r,a)\mid a\in \Zhat/\Zhat^*\ \hbox{and}\ r\in \N \sqcup \Zhat /a\Zhat\}.$$

\begin{proposition}
The induced map $\tilde p : (\bar\N \times \Zhat/\Zhat^*)/_\sim  \to \Omega$ is a homeomorphism.
\end{proposition}

\bp
By \cite[Corollary 5.6]{LR10} the map $\tilde p$ is bijective. To show that it is continuous recall that, by definition, the topology on $\Omega$ is such that $\omega_n\to\omega$ if and only if
$\chf_{{\omega_n}}(x)\to \chf_{{\omega}}(x)$ for every $x\in\N\rtimes\N^\times$.
Assume $(r_n,a_n)\to(r,a)$ in $\bar\N\times(\Zhat/\Zhat^*)$.
We have to show that, for every $(m,k)\in\N\rtimes\N^\times$,
$$
\chf_{p(r_n,a_n)}(m,k)\to \chf_{p(r,a)}(m,k).
$$
If $a\notin k\Zhat/\Zhat^*$ then $a_n\notin k\Zhat/\Zhat^*$ for sufficiently large $n$, because $k\Zhat$ is closed in $\Zhat$. Hence $\chf_{p(r_n,a_n)}(m,k) = \chf_{p(r,a)}(m,k)=0$ for $n$ large enough.

If $a\in k\Zhat/\Zhat^*$, then $a_n\in k\Zhat/\Zhat^*$ for sufficiently large~$n$,
because $k\Zhat$ is open in $\Zhat$, so we may assume that this is true for all $n$. We consider two cases: $r\in\N\subset\bar\N$ and $r\in\Zhat\subset\bar\N$.

\smallskip

{(a)} Assume first that $r\in\N$. Then $r_n=r$ for sufficiently large $n$, so we may assume that $r_n=r$ for all $n$. By definition, $(m,k)\in p(r,a) = A(r,a)$ if and only if $r-m\in k\N$, and the same condition determines whether $(m,k)\in p(r_n,a_n) = A(r_n,a_n)$. Therefore $\chf_{p(r_n,a_n)}(m,k)=\chf_{p(r,a)}(m,k)$.

\smallskip

{(b)} Assume next that $r\in\Zhat$. By decomposing $\{r_n\}_n$ into two subsequences we may assume that either $r_n\in \Zhat$ for all $n$, or $r_n\in\N\subset\bar\N$ for all $n$. Consider separately these two subcases.

\begin{itemize}
\item[(b$_1$)] Suppose $r_n\in\Zhat$ for all $n$. We have $(m,k)\in p(r,a)=B(r_a,a)$ if and only if $r-m\in k\Zhat$, and $(m,k)\in p(r_n,a_n)=B((r_n)_{a_n},a_n)$ if and only if $r_n-m\in k\Zhat$. Since $m+k\Zhat$ is a clopen subset of $\Zhat$ and $r_n\to r$, for $n$ large enough we have $r_n\in m+k\Zhat$ if and only if $r\in m+k\Zhat$, that is, $\chf_{p(r_n,a_n)}(m,k)=\chf_{p(r,a)}(m,k)$.

\smallskip

\item[(b$_2$)]  Finally, suppose $r_n\in\N$ for all $n$. Then $r_n\to\infty$ and $r_n\to r$ in $\Zhat$. As above, $(m,k)\in p(r,a)=B(r_a,a)$ if and only if $r-m\in k\Zhat$, while $(m,k)\in p(r_n,a_n)=A(r_n,a_n)$ if and only if $r_n-m\in k\N$. Since $r_n\to\infty$, for sufficiently large~$n$ the condition $r_n - m\in k\N$ is equivalent to $r_n-m\in k\Z$, which is in turn the same as the condition $r_n-m\in k\Zhat$, because both $r_n$ and $m$ are integers. As in the case (b$_1$) we conclude that $\chf_{p(r_n,a_n)}(m,k) = \chf_{p(r,a)}(m,k)$ for~$n$ large enough.
\end{itemize}
This proves that the map $ p \colon \bar\N \times \Zhat/\Zhat^*  \to \Omega$ is continuous, hence the quotient map $\tilde p$ is a homeomorphism.
\ep

Our parameter space $ (\bar\N \times \Zhat/\Zhat^*)/_\sim = \{(r,a)\mid a\in \Zhat/\Zhat^*\ \hbox{and}\ r\in \bar\N_a\}$ has a natural action of~$\nxnx$ which we describe next.
We have natural actions of $\N$ by translation on itself and
on $\Zhat$, through the canonical embedding; hence we get an action of $\N$ on $\bar\N$ by translation.  It passes to an action of $\N$ on $\bar\N_a$ for each $a\in \Zhat/\Zhat^*$. The image of $r\in\bar\N_a$ under the action of $m\in\N$ will be denoted by $m+r$. Similarly, multiplication by $k\in\N$ on $\Zhat$ induces a homomorphism $\Zhat/a\Zhat\to k\Zhat/ka\Zhat$, which together with multiplication by $k$ on~$\N$  defines a map of $\bar \N_a$ to $\bar\N_{ka}$. Consequently, the image of $r\in\bar\N_a$ under this map will be denoted by $kr\in\bar\N_{ka}$.

\begin{proposition}\label{homeomorphismprop}
The action of $ \N\rtimes\N^\times$ on
the space $(\bar\N \times \Zhat/\Zhat^*)/_\sim =\{(r,a)\mid a\in \Zhat/\Zhat^*\ \hbox{and}\ r\in\bar \N_a\}$,
obtained through the map $\tilde p$ from the action on $\Omega$,  is given by
$$
(m,k)(r,a)=(m+kr,ka)\ \ \hbox{for}\ \ (m,k) \in \nxnx.
$$
\end{proposition}

\bp We first make the following observations on how to recover the coordinates of $(r,a)$ from~$\tilde p(r,a)$ (compare with~\cite[Proposition~5.5]{LR10}):

\begin{itemize}
\item[(i)] $a\Zhat=\cap_{(n,l)}\, l\Zhat$, where the intersection is taken over all elements $(n,l)\in \tilde p(r,a)$;

\item[(ii)] if $r\in\N\subset\bar\N_a$ then $\sup\{n\mid (n,l)\in A(r,a) = \tilde p(r,a) \} = r $;

\item[(iii)]  if $r\in\Zhat/a\Zhat \subset\bar\N_a$ then $\sup\{n\mid (n,l)\in  B(r_a,a) =\tilde p(r,a) \}=\infty$.
\end{itemize}

Suppose $(m,k) \in \nxnx$ and $(r,a) \in \bar\N_a \times \Zhat/\Zhat^*$. Recall that, by definition, the image of the hereditary directed set $\tilde p(r,a) \in \Omega$ under the action of
$(m,k) \in\nxnx$ is the smallest directed hereditary  set -- necessarily of the form $\tilde p(s,b)$
because $\tilde p$ is bijective -- containing $\{(m+kn,kl) \mid (n,l) \in \tilde p (r,a)\}$.
%(one needs to take hereditary closures because in general $\{(m+kn,kl) \mid (n,l) \in \tilde p (r,a)\}$ is not %hereditary itself, cf. \cite[Section 2]{Lpur}).
It is easy to see that
$ \tilde p(m+kr,ka)$ contains all the elements $(m+kn,kl)$ with $(n,l)$ in $\tilde p(r,a)$.
This gives the following inclusions:
\begin{equation}\label{inclusions}
\{(m+kn,kl) \mid (n,l) \in \tilde p (r,a)\} \subset \tilde p (s,b)\subset \tilde p(m+kr,ka).
\end{equation}
From the first inclusion and (i) above,  we get
$$
b\Zhat=\bigcap_{(n',l')\in \tilde p(s,b)}l'\Zhat\subset\bigcap_{(n,l)\in \tilde p(r,a)}  k l \Zhat = ka \Zhat.
$$
Similarly, from the second inclusion we get $ka\Zhat\subset b\Zhat$. Hence $b\Zhat=ka\Zhat$, so that $b=ka \in \Zhat/\Zhat^*$.

\smallskip

To determine $s$ we consider two cases. Assume first that $r\in\N\subset\bar\N_a$. Then by (ii)
and the second inclusion,
$$
\sup\{n\mid (n,l)\in \tilde p(s,ka)\}\leq \sup\{n\mid (n,l)\in \tilde p(m+kr,ka)\}=m+kr.
$$
By (iii) we see that $(s,ka)$ cannot be of type B, so $s\in\N\subset\bar \N_{ka}$, and applying (ii) once again we get $s\leq m+kr$. On the other hand, using the first inclusion and (ii) yet another time we get
$$
\sup\{n\mid (n',l')\in \tilde p(s,ka)\} \geq \sup\{m+kn\mid (n,l)\in \tilde p(r,a)\}=m+kr.
$$
Hence $s=m+kr$.

Assume next that $r\in\Zhat/a\Zhat$. Then by (iii)
$$
\sup\{n\mid (n,l)\in \tilde p(s,ka)\}\ge \sup\{m+kn\mid (n,l)\in \tilde p(r,a)\}=\infty.
$$
By (ii) we see that $(s,ka) $ cannot be of type A, so $s\in \Zhat/ka\Zhat$. Let $(n,l)\in \tilde p(r,a)$. Then, since
$$
(m,k)(n,l)=(m+kn,kl)\in \tilde p(s,ka)\subset \tilde p(m+kr,ka),
$$
both elements $s-(m+kn)$ and $m+kr-(m+kn)$ belong to $kl\Zhat/ka\Zhat$, so that $m+kr-s\in kl\Zhat/ka\Zhat$.
Since the intersection of the groups $kl\Zhat$ over all $(n,l)\in \tilde p(r,a)$ is $ka\Zhat$ by (i), we conclude that $s=m+kr$. This proves that the second inclusion in \eqref{inclusions} is an equality and concludes the proof.
\ep

From now on we shall use the homeomorphism $\tilde p$
to identify our parameter space of equivalence classes
$\{(r,a)\mid a\in \Zhat/\Zhat^*\ \hbox{and}\ r\in\bar \N_a\} $ and the spectrum $ \Omega$.
Following~\cite{LR10} we denote by $\Omega_B\subset\Omega$ the set of points of type B, so
$$
\Omega_B := \{(r,a)\mid a\in \Zhat/\Zhat^*\ \hbox{and}\ r\in\Zhat/a\Zhat\}= p(\Zhat\times(\Zhat/\Zhat^*)).
$$
This set was called the additive boundary in~\cite{BaHLR} and was denoted by $\Omega_{\rm add}$. Note also that the multiplicative boundary~\cite{BaHLR} in our description of $\Omega$ is the set $\Omega_{\rm mult}=\{(r,a)\in\Omega\mid a=0\}\cong\bar\N=\N\sqcup \Zhat$.

\begin{proposition} \label{psemiiso}
The action of $\N\rtimes\N^\times$ on $\Omega_B$ extends to an action of $\Z\rtimes\N^\times$.
 The restriction map $C(\Omega)\to C(\Omega_B)$, $f\mapsto f|_{\Omega_B}$, determines an isomorphism
$$
C(\Omega)\rtimes(\N\rtimes\N^\times)/\langle1-ss^*\rangle\cong C(\Omega_B)\rtimes(\Z\rtimes\N^\times),
$$
where $\langle1-ss^*\rangle$ denotes the closed ideal in $\TT(\nxnx$) generated by the range projection of
the isometry $s = T_{(1,1)}$.
\end{proposition}

\bp
The first claim follows simply on noting that the transformation $(m,k) (r,a) = (m+kr, ka)$ is defined on $\Omega_B$ for every $(m,k) \in \Z\rtimes \nx$. The  homomorphism
$
C(\Omega)\rtimes(\N\rtimes\N^\times)\to C(\Omega_B)\rtimes(\Z\rtimes\N^\times)
$ exists by universality of crossed products.
The isomorphism claim then holds by the results of \cite{ELQ} because $\Omega_B$ is the spectrum of the extra relation $1-ss^*$,
see  \cite[Proposition 3.4]{BaHLR}. We may also prove this directly as follows.

Since the image of $s$ in $C(\Omega_B)\rtimes(\Z\rtimes\N^\times)$ is the unitary corresponding to the invertible element $(1,1)\in\Z\rtimes\N^\times$, the homomorphism
$$
C(\Omega)\rtimes(\N\rtimes\N^\times)/\langle1-ss^*\rangle\to C(\Omega_B)\rtimes(\Z\rtimes\N^\times).
$$
is surjective. To construct the inverse homomorphism it is enough to show that the homomorphism
\begin{equation} \label{efactor}
C(\Omega)\to C(\Omega)\rtimes(\N\rtimes\N^\times)/\langle1-ss^*\rangle
\end{equation}
factors through $C(\Omega_B)$. Indeed, since the representation of $\N\rtimes\N^\times$ by isometries in
$$C(\Omega)\rtimes(\N\rtimes\N^\times)/\langle1-ss^*\rangle$$ defines a representation of $\Z\rtimes\N^\times$, we then
get a covariant pair of representations of $C(\Omega_B)$ and~$\Z\rtimes\N^\times$ in $C(\Omega)\rtimes(\N\rtimes\N^\times)/\langle1-ss^*\rangle$ which defines the required homomorphism $C(\Omega_B)\rtimes(\Z\rtimes\N^\times)\to C(\Omega)\rtimes(\N\rtimes\N^\times)/\langle1-ss^*\rangle$.

Since $s^n(s^n)^*=\chf_{(n,1)\Omega}$, the homomorphism from \eqref{efactor}  factors through $C((n,1)\Omega)$ for every $n\in\N$. Since $\cap_n(n+\bar\N_a)=\Zhat/a\Zhat\subset\bar\N_a$, we have $\cap_n(n,1)\Omega=\Omega_B$. Hence the homomorphism indeed factors through $C(\Omega_B)$.
\ep

 Next we will dilate the action of the semigroup $\Z\rtimes\N^\times$ on $\Omega_B$
to an action of the group $\Q\rtimes\Q_+^*$. To do this we need a larger space, which we define by
$$
\tilde\Omega_B :=\{(r,a)\mid a\in\af/\Zhat^*\ \hbox{and}\ r\in\af/a\Zhat\},
$$
where $\af$ is the ring of finite adeles, that is,  the restricted product of $\Q_p$ with respect to $\Z_p$ over~$p\in\primes$. The space $\tilde\Omega_B$ is a quotient of $\af\times(\af/\Zhat^*)$ and we give it the quotient topology; it contains $\Omega_B$ as a compact open subset. The action of $\Z\rtimes\N^\times$ on $\Omega_B$ extends to an action of $\Q\rtimes\Q^*_+$ on $\tilde\Omega_B$ defined by the same formula as before:
$$
(m,k)(r,a)=(m+kr,ka) \ \ \hbox{for}\ \ (m,k) \in \qxqx.
$$

Since $\af$ is the union of $n^{-1}\Zhat$ over all $n\in\N^\times$, it is easy to see that $\tilde\Omega_B$ is the union of $x^{-1}\Omega_B$ over all $x\in(0,\N^\times)\subset\Z\rtimes\N^\times$. Therefore the dynamical system $(C_0(\tilde\Omega_B),\Q\rtimes\Q^*_+)$ is a minimal dilation of $(C(\Omega_B),\Z\rtimes\N^\times)$ in the sense of~\cite[Theorem 2.1]{Ldil}, which is unique up to canonical isomorphism.
By~\cite[Theorem 2.4]{Ldil} we get the following result.

\begin{proposition} \label{pdilate}
The inclusions $C(\Omega_B) \hookrightarrow C_0(\tilde \Omega_B)$ and $\Z\rtimes\N^\times \hookrightarrow \qxqx$ induce a canonical isomorphism
$$
C(\Omega_B)\rtimes(\Z\rtimes\N^\times)\cong \chf_{\Omega_B}(C_0(\tilde\Omega_B)\rtimes(\Q\rtimes\Q_+^*))\chf_{\Omega_B},
$$
and the projection $\chf_{\Omega_B}$ is full in $C_0(\tilde\Omega_B)\rtimes(\Q\rtimes\Q_+^*)$.
\end{proposition}

\begin{corollary}
The representation of $\N\rtimes\N^\times$ by isometries on $\ell^2(\Z\rtimes\N^\times)$ induces an isomorphism
$$
\TT(\N\rtimes\N^\times)/\langle1-ss^*\rangle\cong\TT(\Z\rtimes\N^\times).
$$
\end{corollary}

\bp
We refer to \cite[Example 3.9]{BaHLR} for a quick proof using the characterization of faithful representations of  the quotient
$\TT(\N\rtimes\N^\times)/\langle1-ss^*\rangle$, which is the {\em additive boundary quotient} discussed there.
We can also use the above realization of $C(\Omega_B)\rtimes(\Z\rtimes\N^\times)$ as a full corner to give the following proof.

Consider the point $\omega_1=(0,\Zhat^*)\in\tilde\Omega_B$, so $\omega_1$ is the image of the point $(0,1)\in\Zhat\times\Zhat$ in~$\Omega_B$. Let $\pi$ be the representation of $C_0(\tilde\Omega_B)\rtimes(\Q\rtimes\Q_+^*)$ on~$\ell^2(\Q\rtimes\Q^*_+)$ induced from the character $C_0(\tilde\Omega_B)\to\C$, $f\mapsto f(\omega_1)$, so
$$
\pi(f)\delta_g=f(g\omega_1)\delta_g,\ \ \pi(u_h)\delta_g=\delta_{hg}.
$$
Since $\Q$ is dense in $\af$, the $(\Q\rtimes\Q^*_+)$-orbit of $\omega_1$ is dense in $\tilde\Omega_B$, so  $\pi |_{C_0(\tilde\Omega_B)}$ is faithful. Since $\qxqx$ is amenable, $\pi$ itself is faithful.
Hence, letting $P=\pi(\chf_{\Omega_B})$, we get a faithful representation of
$$C(\Omega_B)\rtimes(\Z\rtimes\N^\times) =\chf_{\Omega_B}(C_0(\tilde\Omega_B)\rtimes(\Q\rtimes\Q_+^*))\chf_{\Omega_B}$$
on $P\ell^2(\Q\rtimes\Q^*_+)$ such that  $x\mapsto\pi(x)|_{P\ell^2(\Q\rtimes\Q^*_+)}$. Observe next that since $\Q\cap\Zhat=\Z$, the point $g\omega_1$ belongs to~$\Omega_B$ for an element $g\in\Q\rtimes\Q^*_+$ if and only if $g\in\Z\rtimes\N^\times$. Thus $P\ell^2(\Q\rtimes\Q^*_+)=\ell^2(\Z\rtimes\N^\times)$.

To summarize, the representation of $\Z\rtimes\N^\times$ by isometries on $\ell^2(\Z\rtimes\N^\times)$ extends to a faithful representation of $C(\Omega_B)\rtimes(\Z\rtimes\N^\times)$. Since the latter algebra is canonically isomorphic to the algebra $\TT(\N\rtimes\N^\times)/\langle1-ss^*\rangle$ by Proposition~\ref{psemiiso}, we get the result.
\ep

%The algebra $\TT(\N\rtimes\N^\times)$ was described in terms of generators and %relations in~\cite[Theorem~4.1]{LR10}.  As an immediate consequence we  can %retrieve the presentation of $\TT(\Z\rtimes\N^\times)$ obtained in %\cite[Proposition~3.3]{BaHLR}.

%\begin{corollary}
%The Toeplitz algebra $\TT(\Z\rtimes\N^\times)$ is the universal unital %C$^*$-algebra generated by a unitary $u$ and isometries $v_p$, $p\in\primes$, %satisfying the relations
%$$v_pu=u^pv_p,\ \ v_pv_q=v_qv_p,\ \ v_p^*v_q=v_qv_p^*\ \ \hbox{when}\ \ p\ne %q,\ \ v_p^*u^kv_p=0\ \ \hbox{for}\ \ 1\le k<p.$$
%\end{corollary}

The action of $\qxqx$ on $\tilde \Omega_B$ is not free.  Next we characterize
the points having nontrivial stabilizers.
This will be crucial for our discussion of KMS-states
in Section \ref{secphasetrans}.
Recall that in the previous proof we introduced a special point $\omega_1=(0,\Zhat^*)\in\tilde\Omega_B$, which is the image of the point $(0,1)\in\Zhat\times\Zhat$ in $\Omega_B$.

\begin{lemma} \label{rstab}
We have:
\enu{i} if a point  $(r,a)\in\tilde \Omega_B$ has nontrivial stabilizer in $\qxqx$ then either $a_p=0$ for some $p\in\primes$, or $(r,a)$ is on the orbit $(\qxqx ) \omega_1$; the stabilizer of $(m,k)\omega_1$ in $\qxqx$ is $k\Z \times \{1\}$;
\enu{ii} a point  $(r,a)\in\tilde \Omega_B$  has nontrivial stabilizer in $\Z\times\{1\}\subset\qxqx$ if and only if $(r,a)$ is on the orbit $(\qxqx ) \omega_1$; the stabilizer of $(m,k)\omega_1$ in $\Z\times\{1\}$ is $(k\Z\cap\Z)\times\{1\}$.
%$$\{(r,a)\mid a\in\Q^*_+\Zhat^*/\Zhat^*\ \hbox{and}\ r\in\af/a\Zhat\}.$$
\end{lemma}

\begin{proof}
Assume $a_p\ne0$ for all $p\in \primes$. Then $ka\ne a$ for any $k\in\Q^*_+$, $k\ne1$. Hence the stabilizer of~$(r,a)$ is contained in $\Q\times\{1\}$.
It follows that to prove (i) and (ii) it suffices to show that a point~$(r,a)$ (with no restrictions on $a$) has nontrivial stabilizer in $\Q\times\{1\}\subset\qxqx$ if and only if~$(r,a)$ is on the orbit $(\qxqx ) \omega_1$, and the stabilizer of $(m,k)\omega_1$ in $\Q\times\{1\}$ is $k\Z\times\{1\}$.

Assume $(m,1)\ne(0,1)$ stabilizes $(r,a)$. Since $(m,1)(r,a)=(m+r,a)$, this means that $m\in a\Zhat$. It follows that $a$ has the valuation vector of a nonzero rational number,
i.e., all exponents are finite and only finitely many are nonzero, in other words, $a\in\Q^*_+\Zhat^*/\Zhat^*$.
Then $a\Zhat$ is open in $\af$, and since~$\Q$ is dense in $\af$,  we have $\af = \Q + a \Zhat$, so $r\in(\Q+a\Zhat)/a\Zhat$.
Therefore $(r,a)$ lies on the orbit of the
point $\omega_1=(0,\Zhat^*)\in\tilde \Omega_B$.  Since $\Q\cap\Zhat=\Z$, the point $\omega_1$
has stabilizer $\Z \times \{1\}\subset \Q\rtimes\Q^*_+$, and hence the stabilizer of $(m,k)\omega_1$ is $(m,k)(\Z \times \{1\})(m,k)^{-1}=k\Z \times \{1\}$.
\end{proof}

\begin{remark}
Although we will not need it here, we point out that
there is a similar dilation of the action of $\N\rtimes\N^\times$ on the whole space $\Omega$, which
realizes the Toeplitz algebra $\TT(\nxnx)$ as a full  corner in a group crossed product.
Namely, for every~$a\in\af$ we denote by $\bar \Q_a$ the disjoint union $\Q\sqcup(\af/a\Zhat)$; write $\bar\Q$ for $\bar\Q_0$. The space $\bar\Q_a$ depends only on the image of $a$ in $\af/\Zhat^*$, and we define
$$
\tilde\Omega:=\{(r,a)\mid a\in\af/\Zhat^*\ \hbox{and} \ r\in\bar\Q_a\}.
$$
Define a topology on $\bar\Q=\Q\sqcup\af$ similarly to how we defined the topology on $\bar\N=\N\sqcup\Zhat$. The space~$\tilde\Omega$ is a quotient of $\bar\Q\times(\af/\Zhat^*)$, and we give it the quotient topology. Then
$$
\TT(\N\rtimes\N^\times)=C(\Omega)\rtimes(\N\rtimes\N^\times)\cong \chf_{\Omega}(C_0(\tilde\Omega)\rtimes(\Q\rtimes\Q_+^*))\chf_{\Omega}.
$$
\end{remark}

\bigskip

\section{The phase transition on $\TT(\nxnx)$ revisited}\label{secphasetrans}

As in \cite{LR10} we consider the  one-parameter automorphism group $\sigma$ of $\TT(\N\rtimes\N^\times)$ defined by
$$
\sigma_t(T_{(m,k)})=k^{it}T_{(m,k)},
$$
and we recall that a $\sigma$-invariant state $\varphi$ is called a $\sigma$-KMS$_\beta$-state ($\beta\in\R$) if
$$
\varphi(xy)=\varphi(y\sigma_{i\beta}(x))
$$
for any $\sigma$-analytic elements $x,y$. The $\sigma$-KMS$_\beta$-states of $\TT(\N\rtimes\N^\times)$ were classified in~\cite{LR10}. In particular, it was shown that for every $\beta\in[1,2]$ there exists a unique KMS$_\beta$-state $\varphi_\beta$; this implies that each $\varphi_\beta$ is a factor state, but the classification
of their type was left open in \cite{LR10}. Our goal is to relate the measures $\mu_\beta$ defined by $\varphi_\beta$ on $\Omega$ to measures that appeared in the study of the Bost-Connes system~\cite{bos-con}, and then use this, in the next section, to show that $\varphi_\beta$ has type III$_1$. It is possible to do this by adapting the definition of $\mu_\beta$ in~\cite{LR10} to our description of~$\Omega$. Instead, we will use crossed product decompositions from the previous section to construct~$\mu_\beta$ anew.

\smallskip

For any $\sigma$-KMS$_\beta$-state $\varphi$ of $\TT(\nxnx)$ we have $\varphi(ss^*)=\varphi(s^*s)=1$, so $\varphi$ factors through
$$
\TT(\N\rtimes\N^\times)/\langle1-ss^*\rangle\cong C(\Omega_B)\rtimes(\Z\rtimes\N^\times),
$$
and therefore in studying KMS-states of $\TT(\nxnx)$ we can work with the quotient algebra $C(\Omega_B)\rtimes(\Z\rtimes\N^\times)$. We denote the induced dynamics on $C(\Omega_B)\rtimes(\Z\rtimes\N^\times)$ also by $\sigma$. Since
$$
C(\Omega_B)\rtimes(\Z\rtimes\N^\times)=\chf_{\Omega_B}(C_0(\tilde\Omega_B)\rtimes(\Q\rtimes\Q^*_+))\chf_{\Omega_B}
$$
is a groupoid algebra, if $\varphi$ is a $\sigma$-KMS$_\beta$-state on $C(\Omega_B)\rtimes(\Z\rtimes\N^\times)$ and $\mu$ is the probability measure on $\Omega_B$ defined by $\varphi$, then $\mu$ satisfies the scaling condition
\begin{equation}\label{escaling1}
\mu((m,n)Y)=n^{-\beta}\mu(Y)\ \hbox{for any Borel set}\ Y\subset \Omega_B\ \hbox{and any}\ (m,n)\in\Z\rtimes\N^\times,
\end{equation}
see e.g.~the proof of~\cite[Proposition II.5.4]{ren}. Conversely, given such a measure $\mu$ we obtain a $\sigma$-KMS$_\beta$-state by composing the state $\mu_*$ on $C(\Omega_B)$ defined by $\mu$ with the canonical conditional expectation
$C(\Omega_B)\rtimes(\Z\rtimes\N^\times)\to C(\Omega_B)$.

\begin{proposition} \label{psmallbeta}
For every $\beta\ge1$ there exists a unique probability measure $\mu_\beta$ on $\Omega_B$ satisfying the scaling condition \eqref{escaling1}. There are no such measures for $\beta<1$.

Equivalently, for every $\beta\ge1$ there exists a unique $\sigma$-KMS$_\beta$-state on $C(\Omega_B)\rtimes(\Z\rtimes\N^\times)$ which factors through the conditional expectation onto $C(\Omega_B)$. There are no $\sigma$-KMS$_\beta$-states for $\beta<1$.
\end{proposition}

\bp Let $\mu$ be a probability measure on $\Omega_B$.
Denote by $\nu$ the image of $\mu$ under the projection $\Omega_B\to\Zhat/\Zhat^*$ onto the second coordinate. Disintegrating $\mu$ with respect to this projection map we get probability measures $\lambda_a$ on $\Zhat/a\Zhat$ such that
$$
\int_{\Omega_B}f\,d\mu=\int_{\Zhat/\Zhat^*}\left(\int_{\Zhat/a\Zhat}f(r,a)d\lambda_a(r)\right)d\nu(a)\ \ \hbox{for}\ \ f\in C(\Omega_B).
$$
If $\mu$ satisfies condition \eqref{escaling1}, it is in particular $(1,1)$-invariant, hence for $\nu$-a.a. $a\in\Zhat/\Zhat^*$ the measure~$\lambda_a$ is invariant under the action of $\Z$ on $\Zhat/a\Zhat$ by translations. Since $\Z$ is dense in $\Zhat$, it follows that $\lambda_a$ is a Haar measure for $\nu$-a.a. $a$. Denoting the Haar probability measure on $\Zhat$ by $\nu_1$, we conclude that $\mu$ is the image of the measure $\nu_1\times\nu$ on $\Zhat\times(\Zhat/\Zhat^*)$ under the map $p\colon \Zhat\times(\Zhat/\Zhat^*)\to\Omega_B$. Since the Haar measure $\nu_1$ has the property $\nu_1(nY)=n^{-1}\nu_1(Y)$, it follows that $\mu$ satisfies \eqref{escaling1} if and only if $\nu$ satisfies
$$
\nu(nY)=n^{-(\beta-1)}\nu(Y)\ \hbox{for any Borel set}\ Y\subset \Zhat/\Zhat^*\ \hbox{and any}\ n\in\N^\times.
$$
By \cite[Proposition~18(i)]{Ldir} this is exactly the condition that determines a KMS-state on
the symmetric part of the Bost-Connes system at inverse temperature $\beta -1$, since
$\Zhat/\Zhat^*$ is the spectrum of the diagonal~$B(S)$ of the symmetric part.
It is known from~\cite[Proposition 8]{bos-con} and \cite[Proposition~18]{Ldir}
that there are no such measures for $\beta<1$, and  that  there is a unique such measure~$\bar\nu_{\beta-1}$ for every $\beta \geq 1$. Specifically, $\bar\nu_0$ is the delta-measure at $0$, and for $\beta>1$ the measure $\bar\nu_{\beta-1}$ is the product-measure $\prod_{p\in\primes}\bar\nu_{\beta-1,p}$ on $\Zhat/\Zhat^*=\prod_p\Z_p/\Z_p^*$, where $\bar\nu_{\beta-1,p}$ is defined by $\bar\nu_{\beta-1,p}(p^n\Z^*_p)=p^{-n(\beta-1)}(1-p^{-(\beta-1)})$ and $\bar\nu_{\beta-1,p}(0)=0$.
\ep

Using Lemma~\ref{rstab} it is not difficult to check that for $\beta\in[1,2]$ the set of points in $\Omega_B$ with nontrivial stabilizers in $\qxqx$ has $\mu_\beta$-measure zero, see Lemma~\ref{lfree} below. It follows that any KMS$_\beta$-state (for $\beta\in[1,2]$) factors through the conditional expectation onto $C(\Omega_B)$, see e.g.~the proof of \cite[Proposition~1.1]{LLNlat}. Thus the above proposition could be used to classify KMS$_\beta$-states for $\beta\le2$. On the other hand, as we will see soon, for $\beta>2$ the measure~$\mu_\beta$ is concentrated on points with nontrivial stabilizers, and indeed there are KMS$_\beta$-states which do not factor through the conditional expectation onto $C(\Omega_B)$~\cite{LR10}. In the remaining part of the section we will show how to complete the classification of KMS-states. This will not be used in the next section, but will provide a conceptual explanation of the appearance of a circular symmetry of KMS$_\beta$-states for $\beta>2$ observed in~\cite{LR10}.

We can write $C(\Omega_B)\rtimes(\Z\rtimes\N^\times)$ as
$
(C(\Omega_B)\rtimes\Z)\rtimes_\alpha\N^\times.
$
Denote by $u$ the unitary in $C(\Omega_B)\rtimes\Z$ defining the action by $\Z$,
so
\[
ufu^*=f((1,1)^{-1}\cdot)\ \ \hbox{for}\ \ f\in C(\Omega_B),
\]
and  denote by $v_n$ the isometries defining the multiplicative action $\alpha$ of~$\N^\times$ on $C(\Omega_B)\rtimes\Z$, so
$$
\alpha_n(f)=v_nfv_n^*=f((0,n)^{-1}\cdot)\ \ \hbox{for}\ \ f\in C(\Omega_B),\quad \alpha_n(u)=v_nuv_n^*=u^n.
$$

\begin{proposition}\label{proKMScharact}
There is a one-to-one correspondence between $\sigma$-KMS$_\beta$-states on $(C(\Omega_B)\rtimes\Z)\rtimes\N^\times$ and tracial states $\tau$ on $C(\Omega_B)\rtimes\Z$ such that
\begin{equation} \label{escaling}
\tau\circ\alpha_n=n^{-\beta}\tau\ \ \hbox{for all}\ \ n\in\N^\times.
\end{equation}
\end{proposition}

\begin{proof}
The dynamics $\sigma$ is trivial on $C(\Omega_B)\rtimes\Z$, and $\sigma_t(v_n)=n^{it}v_n$. Hence
the result follows immediately by Theorem~\ref{thmKMSsemi} below.
\end{proof}

We next use results from~\cite{nes3} to describe traces on $C(\Omega_B)\rtimes\Z$.

\smallskip

By Lemma~\ref{rstab} the union of periodic orbits for the action of $\Z\times\{1\}$ on $\Omega_B$ is the set $(\Z\rtimes\N^\times)\omega_1$, where $\omega_1=(0,\Zhat^*)$. In particular, it is countable. By  \cite[Corollary 6]{nes3} any trace~$\tau$ on~$C(\Omega_B)\rtimes\nolinebreak\Z$ uniquely decomposes into a sum of two traces $\tau_1$ and $\tau_2$ such that the measure defined by $\tau_1$ on~$\Omega_B$ is zero on $(\Z\rtimes\N^\times)\omega_1$, while the measure defined by $\tau_2$ on $\Omega_B$ is concentrated on $(\Z\rtimes\N^\times)\omega_1$. Furthermore, then $\tau_1=\tau_1\circ E$, where $E\colon C(\Omega_B)\rtimes\Z\to C(\Omega_B)$ is the canonical conditional expectation. We will say that $\tau_1$ is concentrated on nonperiodic points and $\tau_2$ is concentrated on periodic points.

Since the set $(\Z\rtimes\N^\times)\omega_1$ of periodic points
is invariant under the partial action of $\Q\rtimes\Q^*_+$, it is clear that
$\tau_1\circ\alpha_n$ is still concentrated on nonperiodic points, while $\tau_2\circ\alpha_n$ is concentrated on periodic points. Hence, $\tau\circ\alpha_n=n^{-\beta}\tau$ if and only if $\tau_i\circ\alpha_n=n^{-\beta}\tau_i$ for $i=1,2$. So we can consider the two cases $\tau = \tau_1$ and $\tau = \tau_2$ separately.

\smallskip

If $\tau$ is concentrated on nonperiodic points, then $\tau$ factors through the conditional expectation onto $C(\Omega_B)$ and so is determined by a $(\Z\times\{1\})$-invariant probability measure $\mu$ on $\Omega_B$. The trace $\tau=\mu_*\circ E$ satisfies~\eqref{escaling} if and only if
\begin{equation*}\label{escalingmeas}
\mu((0,n)Y)=n^{-\beta}\mu(Y)\ \hbox{for any Borel set}\ Y\subset \Omega_B\ \hbox{and any}\ n\in\N^\times.
\end{equation*}
Together with $(\Z\times\{1\})$-invariance this means that $\mu$ satisfies the scaling condition~\eqref{escaling1}. We have described such measures in Proposition~\ref{psmallbeta}.
Note once again, however, that by Remark~\ref{rmksupport} below, only for $\beta\in[1,2]$ the measure $\mu_\beta$ is concentrated on nonperiodic points.

\smallskip

Turning to traces concentrated on periodic points, for each $m\in \nx$ let $\omega_m=(0,m)\omega_1$. The $(\Z\times\{1\})$-orbit of $\omega_m$ consists of $m$ points,
and every periodic point is on one of these orbits.
From the discussion leading to  \cite[Corollary 6]{nes3}, we know that for each $m\in \nx$ there is a conditional expectation $E_{\omega_m}\colon C(\Omega_B)\rtimes\Z\to C^*(\Z)$
defined by averaging over the orbit of $\omega_m$:
$$
E_{\omega_m}(f u^k)=\frac{1}{m}\left(\sum^{m-1}_{l=0}f((l,1)\omega_m)\right)u^k
=\frac{1}{m}\left(\sum^{m-1}_{l=0}f((l,m)\omega_1)\right)u^k.
$$
Suppose $\lambda_m$  is an {\em $m$-rotation invariant measure} on $\T$,
i.e., a measure that is invariant under the rotation by $2\pi/m$. Denote by $\lambda_{m*}$ the positive linear functional on $C^*(\Z)\cong C(\T)$ defined by~$\lambda_m$. Then
$\lambda_{m*}\circ E_{\omega_m}$ is a (nonnormalized) trace
concentrated on the orbit of $\omega_m$. If the sequence $\{\lambda_m\}_{m\in \nx}$
is normalized by $\sum_m\lambda_m(\T)=1$, then
\begin{equation} \label{esingtrace}
\tau=\sum^\infty_{m=1}\lambda_{m*}\circ E_{\omega_m}
\end{equation}
is a tracial state concentrated on periodic points.
By \cite[Corollary 6]{nes3}, all tracial states concentrated on periodic points are of this form,
and the decomposition is unique.

\begin{proposition} \label{plargebeta} For every $\beta>2$ the map $\tau\mapsto\zeta(\beta-1)\lambda_1$ is an affine bijection between tracial states on $C(\Omega_B)\rtimes\Z$ of the form~\eqref{esingtrace} satisfying the scaling condition \eqref{escaling} and probability measures on the circle. For $\beta\le2$ there are no such traces.
\end{proposition}

\bp As $\alpha_n(f)=f((0,n^{-1})\cdot)$ for $f\in C(\Omega_B)$, and $(0,n^{-1})(l,m)\omega_1=(l/n,m/n)\omega_1$ lies in $\Omega_B$ only if $n|l$ and $n|m$, by definition of $E_{\omega_m}$ we immediately get
\begin{equation} \label{econd}
(E_{\omega_m}\circ\alpha_n)(f)=\begin{cases}n^{-1}E_{\omega_{m/n}}(f), &\hbox{if}\ \ n|m,\\0, &\hbox{otherwise}.\end{cases}%\ \ \hbox{for}\ \ f\in C(\Omega_B).
\end{equation}
Recall that $\alpha_n(u)=u^n$. For a measure $\lambda$ on $\T$ denote by $\lambda^n$ the image of $\lambda$ under the map $\T\to\T$, $z\mapsto z^n$. Consider a trace $\tau=\sum^\infty_{m=1}\lambda_{m*}\circ E_{\omega_m}$.
Then by \eqref{econd}, for $f\in C(\Omega_B)$ and $k\in\Z$ we get
$$
(\tau\circ\alpha_n)(fu^k)=
\sum^\infty_{m=1}\lambda_{m*}(E_{\omega_m}(\alpha_n(f))u^{kn})
=\sum^\infty_{m=1}\lambda^n_{m*}(E_{\omega_m}(\alpha_n(f))u^k)
=\frac{1}{n}\sum^\infty_{m=1}(\lambda^n_{nm*}\circ E_{\omega_m})(fu^{k}).
$$
Therefore, by uniqueness of the decomposition \eqref{esingtrace},  $\tau\circ\alpha_n=n^{-\beta}\tau$ if and only if $\lambda_m=n^{\beta-1}\lambda^n_{nm}$ for all $m\ge1$.

To see that the sequence $\{\lambda_m\}_{m\in \nx}$ is determined by its first term,
notice first that the map $\lambda\mapsto\lambda^n$ is a bijection between $n$-rotation invariant and all measures on $\T$. Hence for any measure $\lambda_1$ on $\T$ and each $n\in \nx$ there is a unique $n$-rotation invariant measure $\lambda_n$ such that $\lambda_1=n^{\beta-1}\lambda^n_{n}$. Then $\lambda_m^m=m^{-(\beta-1)}\lambda_1=n^{\beta-1}\lambda^{nm}_{nm}$, whence $\lambda_m=n^{\beta-1}\lambda^n_{nm}$ for all $m,n\ge1$. Thus the map $\{\lambda_m\}_m\mapsto\lambda_1$ is a bijection between sequences of $m$-rotation invariant measures $\lambda_m$ such that $\lambda_m=n^{\beta-1}\lambda^n_{nm}$ and all measures on the circle. Finally, for the sum $\sum_m\lambda_{m*}\circ E_{\omega_m}$ to define a tracial state we need the normalization condition $\sum_m\lambda_m(\T)=1$, that is,
$
\sum^\infty_{m=1}m^{-(\beta-1)}\lambda_1(\T)=1$.
This is equivalent to the conditions $\beta>2$ and $\lambda_1(\T)=\zeta(\beta-1)^{-1}$.
\ep

\begin{remark}\label{rmksupport}
If $\varphi$ is a KMS$_\beta$-state and $\mu$ is the measure on $\Omega_B$ defined by $\varphi$, then $\mu$ satisfies the scaling condition~\eqref{escaling}, so by Proposition~\ref{psmallbeta} we have $\beta\ge1$ and $\mu=\mu_\beta$. Since by Proposition~\ref{plargebeta} there are KMS$_\beta$-states concentrated on periodic points only for $\beta>2$, we conclude that the measure~$\mu_\beta$ is concentrated on periodic points for $\beta>2$ and it is concentrated on nonperiodic points for $\beta\in[1,2]$. This is, of course, easy to check directly using the definition of $\mu_\beta$ as a product measure.

Note also that if $\beta >2$ and we choose $\lambda_1$ proportional to the Lebesgue measure, then the measures~$\lambda_m$ are proportional to the Lebesgue measure for all $m$. Then $\tau=\sum_m\lambda_{m*}\circ E_{\omega_m}$ factors through the conditional expectation~$E\colon C(\Omega_B)\rtimes\Z\to C(\Omega_B)$, and the corresponding KMS$_\beta$-state on $C(\Omega_B)\rtimes(\Z\rtimes\nx)$ factors through the conditional expectation $C(\Omega_B)\rtimes(\Z\rtimes\nx)\to C(\Omega_B)$, so it is the state given by Proposition~\ref{psmallbeta}.
\end{remark}

Summarizing our discussion of KMS-states we recover the following result of Laca and Raeburn, \cite[Theorem~7.1]{LR10}

\begin{theorem}
For the C$^*$-dynamical system $(\TT(\N\rtimes\nx),\sigma)$ we have:
\enu{i} for $\beta<1$ there are no KMS$_\beta$-states;
\enu{ii} for every $\beta\in[1,2]$ there is a unique KMS$_\beta$-state;
\enu{iii} for every $\beta>2$ there is an affine homeomorphism between the simplex of KMS$_\beta$-states and the simplex of probability measures on $\T$.
\end{theorem}

\bigskip

\section{Type III$_1$ KMS-states and an ergodic action of $\qxqx$}\label{stype}

It was shown in~\cite{LR10} that all extremal KMS$_\beta$-states on $\TT(\N\rtimes\N^\times)$ for $\beta>2$ are of type I. Our goal now is to show that for each $\beta\in[1,2]$ the unique KMS$_\beta$-state $\varphi_\beta$ is of type III$_1$.

As observed in Section~\ref{secphasetrans}, the state $\varphi_\beta$ factors through
$$
\TT(\Z\rtimes\N^\times)= C(\Omega_B)\rtimes(\Z\rtimes\N^\times)
=\chf_{\Omega_B}(C_0(\tilde\Omega_B)\rtimes(\Q\rtimes\Q^*_+))\chf_{\Omega_B}
$$
and is determined by a probability measure $\mu_\beta$ on $\Omega_B$ satisfying the scaling condition~\eqref{escaling1}. It easily follows that $\mu_\beta$ extends uniquely to a measure $\tilde\mu_\beta$ on $\tilde\Omega_B$ such that
\begin{equation} \label{escaling2}
\tilde\mu_\beta((m,k)Y)=k^{-\beta}\tilde\mu_\beta(Y)\ \hbox{for any Borel set}\ Y\subset \tilde\Omega_B\ \hbox{and any}\ (m,k)\in\Q\rtimes\Q^*_+,
\end{equation}
see, for instance, the argument in the proof of~\cite[Lemma 2.2]{LLNcm}. Denote by $E_{\Q\rtimes\Q^*_+}$ the canonical conditional expectation $C_0(\tilde\Omega_B)\rtimes(\Q\rtimes\Q^*_+)\to C_0(\tilde\Omega_B).$
Then by construction $\varphi_\beta$ is the restriction of the weight $\tilde\mu_{\beta*}\circ E_{\Q\rtimes\Q^\times_+}$ on
$C_0(\tilde\Omega_B)\rtimes(\Q\rtimes\Q^*_+)$ to the corner $\chf_{\Omega_B}(C_0(\tilde\Omega_B)\rtimes(\Q\rtimes\Q^*_+))\chf_{\Omega_B}.$
It follows that the von Neumann algebra $\pi_{\varphi_\beta}(\TT(\N\rtimes\N^\times))''$ generated by $\TT(\N\rtimes\N^\times)$ in the GNS-representation defined by $\varphi_\beta$ is isomorphic to the corner
$
\chf_{\Omega_B}(L^\infty(\tilde\Omega_B,\tilde\mu_\beta)\rtimes(\Q\rtimes\Q^*_+))\chf_{\Omega_B},
$
where the crossed product is understood in the von Neumann algebra sense.

\begin{lemma} \label{lfree}
For every $\beta\in[1,2]$ the action of $\qxqx$ on $(\tilde\Omega_B,\tilde\mu_\beta)$ is ergodic and essentially free.
\end{lemma}

\bp Since $\varphi_\beta$ is a unique, hence extremal, KMS$_\beta$-state, the von Neumann algebra $\pi_{\varphi_\beta}(\TT(\N\rtimes\N^\times))''$ is a factor. As $\chf_{\Omega_B}$ is a full projection in $C_0(\tilde\Omega_B)\rtimes(\Q\rtimes\Q^*_+)$, it follows that $L^\infty(\tilde\Omega_B,\tilde\mu_\beta)\rtimes(\Q\rtimes\Q^*_+)$ is also a factor, hence the action of $\Q\rtimes\Q^*_+$ on $(\tilde\Omega_B,\tilde\mu_\beta)$ is ergodic (this will also follow from the proof of Theorem~\ref{ttype} below).

To show that the action is essentially free, recall from the proof of Proposition~\ref{psmallbeta} that the measure~$\mu_\beta$ is the image under the map $$p\colon\Zhat\times(\Zhat/\Zhat^*)\to\Omega_B$$ of the measure $\nu_1\times\bar\nu_{\beta-1}$, where $\nu_1$ is the Haar probability measure on $\Zhat$ and $\bar\nu_{\beta-1}$ is the measure defining the KMS-state of the symmetric part of the Bost-Connes system at inverse temperature~$\beta-1$. Let $\tilde\nu_1$ be the Haar measure on $\af$ normalized so that $\tilde\nu_1(\Zhat)=1$. Similarly to the construction of~$\tilde\mu_\beta$ we can uniquely extend $\bar\nu_{\beta-1}$ to a measure $\tilde{\bar\nu}_{\beta-1}$ on $\af/\Zhat^*$ such that
\begin{equation} \label{escaling3}
\tilde{\bar\nu}_{\beta-1}(kY)=k^{-(\beta-1)}\tilde{\bar\nu}_{\beta-1}(Y)\ \hbox{for any Borel set}\ Y\subset \af/\Zhat^*\ \hbox{and any}\ k\in\Q^*_+.
\end{equation}
Then the measure space $(\tilde\Omega_B,\tilde\mu_{\beta})$ is a quotient of $(\af\times(\af/\Zhat^*),\tilde\nu_1\times\tilde{\bar\nu}_{\beta-1})$. Consider now two cases.

Assume $\beta=1$. Since $\tilde{\bar\nu}_0$ is the delta-measure at zero,
the measure space $(\tilde\Omega_B,\tilde\mu_1)$ coincides, modulo sets of measure zero, with $(\af,\tilde\nu_1)$. The action of $\qxqx$ on $\af$ is given by $(m,k)r=m+kr$. It follows that the set of points with nontrivial stabilizers is $\Q\subset\af$. Clearly, $\tilde\nu_1(\Q)=0$.

Assume now that $\beta\in(1,2]$. By Lemma~\ref{rstab} the set of points in $\tilde\Omega_B$ with nontrivial stabilizers is contained in the union of the sets $\{(r,a)\mid a_p=0\ \hbox{for some}\ p\in\primes\}$ and $(\Q\rtimes\Q^*_+)\omega_1$, where $\omega_1=(0,\Zhat^*)$. The scaling condition~\eqref{escaling3}, applied to elements $k = p\in\primes$, easily implies that the set of points $a\in\af/\Zhat^*$ with $a_p=0$ has $\tilde{\bar\nu}_{\beta-1}$-measure zero
for each $p\in\primes$. Hence the $\tilde\mu_\beta$-measure of the set $\{(r,a)\mid a_p=0\ \hbox{for some}\ p\in\primes\}$ is zero. On the other hand, as we pointed out in Remark~\ref{rmksupport}, the $\mu_\beta$-measure of $(\Z\rtimes\N^\times)\omega_1$ is zero, from which it follows that the $\tilde\mu_\beta$-measure of $(\Q\rtimes\Q^*_+)\omega_1$ is zero as well.
\ep

We want to compute the flow of weights of the factor $\pi_{\varphi_\beta}(\TT(\N\rtimes\N^\times))''$. Since it does not change under reduction, this is the same as computing the flow of weights of $L^\infty(\tilde\Omega_B,\tilde\mu_\beta)\rtimes(\Q\rtimes\Q^*_+)$. For transformation group von Neumann algebras  it can be described as follows, see~\cite{CT}.

Assume a discrete group $G$ acts freely ergodically by nonsingular transformations on a measure space~$(X,\mu)$.
Denote by $\lambda_\infty$ the Lebesgue measure on $\R^*_+$. We have two commuting actions, of $G$ and of $\R$, on the product space $(\R^*_+\times X,\lambda_\infty\times\mu)$:
$$
g(t,x)=\left(\frac{dg\mu}{d\mu}(gx)t,gx\right)\quad\text{ for } g\in G,\quad \text{and} \quad
s(t,x)=(e^{-s}t,x)\quad\text{ for } s\in\R,
$$
where $g\mu$ is the measure defined by $g\mu(Y)=\mu(g^{-1}Y)$. The flow of weights of $L^\infty(X,\mu)\rtimes G$ is
the flow induced by the above action of $\R$ on the measure-theoretic quotient of $(\R^*_+\times X,\lambda_\infty\times\mu)$ by the action of $G$. A factor is said to be of type III$_1$ if its flow of weights is trivial. So $L^\infty(X,\mu)\rtimes G$ is of type III$_1$ if and only if the action of $G$ on
$(\R^*_+\times X,\lambda_\infty\times\mu)$ is ergodic.

Recall also that there exists a unique injective factor of type III$_1$ with separable predual.

\begin{theorem} \label{ttype}
For every $\beta\in[1,2]$ the von Neumman algebra $\pi_{\varphi_\beta}(\TT(\N\rtimes\N^\times))''$ is isomorphic to the injective factor of type III$_1$ with separable predual.
\end{theorem}

\bp As observed in the preceding discussion, we just have to prove the same statement for $L^\infty(\tilde\Omega_B,\tilde\mu_\beta)\rtimes(\Q\rtimes\Q^*_+)$. This crossed product is obviously injective since $\Q\rtimes\Q^*_+$ is amenable.

In view of the scaling property~\eqref{escaling2}, we have constant Radon-Nikodym derivatives
$$\frac{d(m,k)\tilde \mu_\beta}{d\tilde \mu_\beta}((m,k)\omega) =  k^\beta, \ \ (m,k) \in \qxqx\ \ \text{and}\ \ \omega \in \tilde\Omega_B,$$
so to prove that $L^\infty(\tilde\Omega_B,\tilde\mu_\beta)\rtimes(\Q\rtimes\Q^*_+)$ has type III$_1$ we have to show that the action of $\Q\rtimes\Q^*_+$ on $(\R^*_+\times\tilde\Omega_B,\lambda_\infty\times\tilde\mu_\beta)$ defined by
$$
(m,k)(t,\omega)=(k^{\beta}t,(m,k)\omega),
$$
is ergodic. Since the map $\R^*_+\to\R^*_+$, $t\mapsto t^\beta$, is a measure class preserving isomorphism, we can instead consider the action of $\Q\rtimes\Q^*_+$ on $(\R^*_+\times\tilde\Omega_B,\lambda_\infty\times\tilde\mu_\beta)$ defined by
$$
(m,k)(t,\omega)=(kt,(m,k)\omega).
$$
By the proof of Lemma~\ref{lfree} the measure space $(\tilde\Omega_B,\tilde\mu_{\beta})$ is a quotient of $(\af\times(\af/\Zhat^*),\tilde\nu_1\times\tilde{\bar\nu}_{\beta-1})$. Hence to prove ergodicity of the action of $\Q\rtimes\Q^*_+$ on $(\R^*_+\times\tilde\Omega_B,\lambda_\infty\times\tilde\mu_\beta)$ it is enough to prove that the action of $\Q\rtimes\Q^*_+$ on $(\R^*_+\times\af\times(\af/\Zhat^*),\lambda_\infty\times\tilde\nu_1\times\tilde{\bar\nu}_{\beta-1})$ defined by
$
(m,k)(t,x,y)=(kt,m+kx,ky)
$
is ergodic.

Since $\Q$ is dense in $\af$, it acts ergodically by translations on $(\af,\tilde\nu_1)$. It follows that any $(\Q\times\{1\})$-invariant measurable function on $\R^*_+\times\af\times(\af/\Zhat^*)$ does not depend on the second coordinate. Therefore we just have to prove that the action of $\Q^*_+$ on $(\R^*_+\times(\af/\Zhat^*),\lambda_\infty\times\tilde{\bar\nu}_{\beta-1})$ defined by
$
k(t,y)=(kt,ky),
$
is ergodic.

For $\beta=1$ this is obvious, since $\tilde{\bar\nu}_0$ is the delta-measure at zero and $\Q^*_+$ is dense in $\R^*_+$. While for $\beta\in(1,2]$ this ergodicity is equivalent to the fact that the KMS-states of the symmetric part of the Bost-Connes system at inverse temperatures in the region $(0,1]$ are of type III$_1$. The result goes back to~\cite{bla}, see~\cite{nes2} for more details.
\ep

\begin{remark}Denote by $\tilde\nu_{\beta-1}$ the $\Zhat^*$-invariant measure on $\af$ such that its image under the map $\af\to\af/\Zhat^*$ is $\tilde{\bar\nu}_{\beta-1}$. It is known~\cite{Ldir,nes} that this measure for $\beta\in[1,2]$ defines the unique KMS-state of the Bost-Connes system at inverse temperature $\beta-1$. For $\beta\in(1,2]$ these states are of type~III$_1$~\cite{bos-con,nes}, equivalently, the action of $\Q^*_+$ on $(\R^*_+\times\af,\lambda_\infty\times\tilde{\nu}_{\beta-1})$ defined by
$
k(t,y)=(kt,ky),
$
is ergodic. Following the same argument as in the proof of the above theorem
for the action of $\qxqx$ on $\af \times \af$ defined by $(m,k)(x,y) = (m+kx, ky)$,
we then get the following result: the crossed product $L^\infty(\af\times\af,\tilde\nu_1\times\tilde\nu_{\beta-1})\rtimes(\Q\rtimes\Q^*_+)$ is a factor of type III$_1$ for every $\beta\in[1,2]$.
\end{remark}

\bigskip

\section{KMS-states on crossed products by abelian semigroups} \label{scrossed}

Here we state and prove our characterization of
KMS-states on the crossed product of an algebra by a semigroup of endomorphisms
in terms of scaling traces on the algebra. This type of result is commonplace
in the study of dynamical systems based on semigroup crossed products and goes back to~\cite{beh}.
The version we give here is tailored for our application in Proposition~\ref{proKMScharact},
but the context of this section is more general, and the results presented here are independent from the previous sections.

\smallskip

Let $\SSS$ be an abelian semigroup with cancelation and identity element $e$. Assume $\SSS$ acts by endomorphisms~$\alpha_x$, $x\in\SSS$, on a unital C$^*$-algebra $A$, and $\alpha_e=\operatorname{id}$. Denote by $\iota\colon A\to A\rtimes\SSS$ the canonical homomorphism, and by $v_x$ the isometries in $A\rtimes\SSS$ implementing $\alpha_x$ via
$\iota(\alpha_x(a)) = v_x \iota(a) v_x^*$.  Recall that by dilating the system one can conclude that $\ker\iota=\overline{\cup_x\ker\alpha_x}$, see e.g.~\cite{Ldil}, especially~\cite[Remark 2.5]{Ldil}.

Assume further that we are given an injective homomorphism $\SSS\to\R_+^*$, $x\mapsto N_x$, and define a one-parameter automorphism group $\sigma$ of the crossed product $A\rtimes\SSS$ by
$$
\sigma_t(\iota(a))=\iota(a)\ \ \hbox{for}\ a\in A,\ \ \sigma_t(v_x)=N_x^{it}v_x\ \ \hbox{for}\ x\in\SSS.
$$
%The following result is similar to~\cite[Theorem 12]{Ldir} but the %construction is different and we stress that we make no assumptions on the %endomorphisms $\alpha_x$; consequently we have to deal with the
%possibility that $A$ may be a proper subalgebra of the fixed point algebra %under the dual action. On the other hand, our method of proof requires $N$ to %be injective.

\begin{theorem} \label{thmKMSsemi}
For every $\beta\in\R$ there is a one-to-one correspondence between $\sigma$-KMS$_\beta$-states on~$A\rtimes\nolinebreak\SSS$ and tracial states $\tau$ on $A$ such that $\tau\circ\alpha_x=N_x^{-\beta}\tau$ for every $x\in\SSS$. Explicitly, the state~$\varphi$ corresponding to $\tau$ is determined by
\begin{equation} \label{eA1}
\varphi(v_x^*\iota(a)v_y)=\begin{cases}N_x^\beta\tau(a\alpha_x(1)),&\hbox{if}\ \ x=y,\\0,&\hbox{otherwise}.\end{cases}
\end{equation}
\end{theorem}

\bp The elements $v_x^*\iota(a)v_y$ span a dense $*$-subalgebra of $A\rtimes\SSS$, since the product of two such elements is again an element of the same form:
\begin{equation} \label{eA2}
v_x^*\iota(a)v_yv_s^*\iota(b)v_t=v_x^*\iota(a)v_s^*v_yv_sv_s^*\iota(b)v_t
=v_{xs}^*\iota(\alpha_s(a)\alpha_{ys}(1)\alpha_y(b))v_{yt}.
\end{equation}
Therefore any state $\varphi$ on $A\rtimes\SSS$ is determined by its values on $v_x^*\iota(a)v_y$. If $\varphi$ is $\sigma$-KMS$_\beta$, then $\varphi$ is tracial on $\iota(A)\subset(A\rtimes\SSS)^\sigma$, so that $\tau:=\varphi\circ\iota$ is a tracial state on $A$. If $x\ne y$ then $N_x^{-1}N_y\neq 1$ and $\varphi(v_x^*\iota(a)v_y)=0$ by $\sigma$-invariance. On the other hand, using the KMS-condition we get
$$
\varphi(v_x^*\iota(a)v_x)=N_x^\beta\varphi(\iota(a)v_xv_x^*)=N_x^\beta\varphi(\iota(a\alpha_x(1)))
=N_x^\beta\tau(a\alpha_x(1)).
$$
Thus $\varphi$ is completely determined by $\tau$. Furthermore, using the KMS-condition once again we get
$$
\tau(\alpha_x(a))=\varphi(v_x\iota(a)v_x^*)=N_x^{-\beta}\varphi(\iota(a)v_x^*v_x)=N_x^{-\beta}\tau(a),
$$
so that $\tau\circ\alpha_x=N_x^{-\beta}\tau$.

\smallskip

Conversely, assume $\tau$ is a tracial state on $A$ such that $\tau\circ\alpha_x=N_x^{-\beta}\tau$. We have to show that it determines a KMS$_\beta$-state $\varphi$ by formula \eqref{eA1}.

Let us first assume that a state $\varphi$ satisfying \eqref{eA1} exists and check the KMS$_\beta$-condition. We have to show that
$$
\varphi(v_x^*\iota(a)v_yv_s^*\iota(b)v_t)=N_x^{\beta}N_y^{-\beta}\varphi(v_s^*\iota(b)v_tv_x^*\iota(a)v_y).
$$
By \eqref{eA2} the left hand side equals
$$
\varphi(v_{xs}^*\iota(\alpha_s(a)\alpha_{ys}(1)\alpha_y(b))v_{yt})=\delta_{xs,yt}N_{xs}^\beta
\tau(\alpha_s(a)\alpha_{ys}(1)\alpha_y(b)\alpha_{xs}(1)),
$$
and, similarly, the right hand side equals
$$
N_x^{\beta}N_y^{-\beta}\varphi(v_{xs}^*\iota(\alpha_x(b)\alpha_{xt}(1)\alpha_t(a))v_{yt})
=\delta_{xs,yt}N_x^{\beta}N_y^{-\beta}N_{xs}^\beta\tau(\alpha_x(b)\alpha_{xt}(1)\alpha_t(a)\alpha_{xs}(1)).
$$
Therefore we have to check, assuming $xs=yt$, that
$$
N_y^\beta\tau(\alpha_s(a)\alpha_{ys}(1)\alpha_y(b)\alpha_{xs}(1))=
N_x^{\beta}\tau(\alpha_x(b)\alpha_{xt}(1)\alpha_t(a)\alpha_{xs}(1)).
$$
By the scaling property of $\tau$ the left hand side equals
$$
N_{xy}^\beta\tau(\alpha_{xs}(a)\alpha_{xys}(1)\alpha_{xy}(b)\alpha_{xxs}(1)),
$$
while the right hand side equals
$$
N_{xy}^{\beta}\tau(\alpha_{xy}(b)\alpha_{xyt}(1)\alpha_{yt}(a)\alpha_{xys}(1))
=N_{xy}^{\beta}\tau(\alpha_{yt}(a)\alpha_{xys}(1)\alpha_{xy}(b)\alpha_{xyt}(1)),
$$
and using that $xs=yt$ we see that these expressions indeed coincide.

\smallskip

It remains to show that a state $\varphi$ satisfying \eqref{eA1} exists. We will define it in three steps.

\smallskip

For every $x\in\SSS$ the set $v_x^*\iota(A)v_x$ is the image of the C$^*$-subalgebra $\alpha_x(1)A\alpha_x(1)$
of~$A$ under the $*$-homomorphism $v_x^*\iota(\cdot)v_x$ with kernel $\alpha_x(1)A\alpha_x(1)\cap\ker\iota$. As we remarked earlier, $\ker\iota=\overline{\cup_y\ker\alpha_y}$. On the other hand, the scaling condition on $\tau$ implies that $\ker\tau\supset\ker\alpha_y$ for every~$y$. Therefore $\ker\iota\subset\ker\tau$. It follows that there exists a tracial positive functional $\psi_x$ on $v_x^*\iota(A)v_x$ such that
$$
\psi_x(v_x^*\iota(a)v_x)=N_x^\beta\tau(\alpha_x(1)a\alpha_x(1))\ \ \hbox{for}\ \ a\in A.
$$
The functional $\psi_x$ is a state because $\psi_x(1)=N_x^\beta\tau(\alpha_x(1))=1$.

\smallskip

We need to show next that the collection $\{\psi_x\}_{x\in \SSS}$ is coherent in the sense that if $x=yz$, then $v_y^*\iota(A)v_y\subset v_x^*\iota(A)v_x$ and $\psi_x=\psi_y$ on $v_y^*\iota(A)v_y$. Indeed,
$$
v_y^*\iota(a)v_y=v_x^*v_z\iota(a)v_z^*v_x=v_x^*\iota(\alpha_z(a))v_x
$$
and
\begin{align*}
\psi_x(v_y^*\iota(a)v_y)&=\psi_x(v_x^*\iota(\alpha_z(a))v_x)=N_x^\beta\tau(\alpha_z(a)\alpha_x(1))=
N_x^\beta N_y^{\beta}\tau(\alpha_{yz}(a)\alpha_{xy}(1))\\
&=N_x^\beta N_y^{\beta}\tau(\alpha_x(a\alpha_{y}(1)))
=N_y^{\beta}\tau(a\alpha_{y}(1))=\psi_y(v_y^*\iota(a)v_y).
\end{align*}
It follows that the collection $\{\psi_x\}_{x\in \SSS}$ of tracial states  defines a tracial state $\psi$ on the C$^*$-subalgebra $F:=\overline{\cup_xv_x^*\iota(A)v_x}$ of~$A\rtimes\SSS$.

\smallskip

Finally, there is a canonical conditional expectation $E\colon A\rtimes\SSS\to F$ such that $E(v_x^*\iota(a)v_y)=0$ if~$x\ne y$. Indeed, the semigroup $\SSS$ embeds into a discrete abelian group $G$ and the universal property gives a dual action~$\gamma$ of $\hat G$ on $A\rtimes\SSS$, see  \cite[Section 4]{mur}
for the details. By averaging over the orbits of~$\gamma$ we obtain the required conditional expectation $E\colon A\rtimes\SSS\to (A\rtimes\SSS)^\gamma=F$. Using $E$ we extend~$\psi$ to a state $\varphi$ on $A\rtimes\SSS$, which satisfies~\eqref{eA1} by construction.
\ep

\begin{remark}\mbox{\ }
\enu{i} The above construction of $\varphi$ is similar to the construction of KMS-states on Cuntz-Pimsner-Toeplitz algebras~\cite{LN}. There one also starts with a trace on the coefficient algebra, extends it to
a state on the usually bigger algebra of gauge-invariant elements, and finally extends that to a state on the whole algebra by means of the canonical conditional expectation.
\enu{ii} In view of~\cite[Theorem 12]{Ldir} and~\cite[Theorem~2.1]{LN} it is natural to ask whether the above result remains true if instead of injectivity of $N$ we require $\beta>0$ and $N_x>1$ for all $x\ne e$. To see that this is not the case consider a dynamical system $(A,\SSS,\alpha)$, an injective $N$ with $N_x>1$ for $x\ne e$ and assume that for some $\beta>0$ there exists a unique $\sigma$-KMS$_\beta$-state; for example, we can take the Bost-Connes system and any $\beta\in(0,1]$, so $A=C(\hat\Z)$, $\SSS=\nx$ and $N_x=x$. Consider the subsemigroup $\tilde\SSS$ of $\Z\times\SSS$ consisting of the unit and elements of the form $(n,x)$ with $n\in\Z$ and $x\ne e$. Using the projection $\tilde\SSS\to\SSS$ define a homomorphism $\tilde N\colon\SSS\to\R^*_+$ and an action~$\tilde\alpha$ of $\tilde\SSS$ on~$A$. Then $\tilde N_x>1$ for all $x\in\tilde\SSS\setminus\{e\}$. Let $\tilde\sigma$ be the dynamics on $A\rtimes_{\tilde\alpha}\tilde\SSS$ defined by $\tilde N$. We have an obvious surjective $*$-homomorphism $\pi\colon A\rtimes_{\tilde\alpha}\tilde\SSS\to C(\T)\otimes(A\rtimes_{\alpha}\SSS)$, mapping $v_{(n,x)}$ into $u^n\otimes v_x$, where $u\in C(\T)$ is the canonical unitary generator. Then $\pi\circ\tilde\sigma_t=(\operatorname{id}\otimes\sigma_t)\circ\pi$. It follows that if $\varphi$ is the unique $\sigma$-KMS$_\beta$-state on $A\rtimes_\alpha\SSS$ and $\psi$ is a state on $C(\T)$ then $(\psi\otimes\varphi)\circ\pi$ is a $\tilde\sigma$-KMS$_\beta$-state on $A\rtimes_{\tilde\alpha}\tilde\SSS$. Thus we have a unique tracial state $\tau$ on $A$ such that $\tau\circ\tilde\alpha_x=\tilde N_x^{-\beta}\tau$, but there are infinitely many $\tilde\sigma$-KMS$_\beta$-states on $A\rtimes_{\tilde\alpha}\tilde\SSS$.
\end{remark}

\end{document}